\documentclass[a4paper,12pt]{article}
\usepackage{geometry,latexsym,amssymb,amsthm,amsmath,color}
\usepackage[latin5]{inputenc}
\usepackage{enumerate}
\usepackage{enumitem}
\usepackage[T1]{fontenc}
\usepackage{authblk}
\usepackage[all]{xy}
\usepackage{palatino}
\usepackage{indentfirst}
\usepackage{setspace}
\newtheorem{theorem}{Theorem}[section]
\newtheorem*{theorem A}{Theorem A}
\newtheorem*{theorem B}{N\"olker's Theorem}

\newtheorem{corollary}{Corollary}[section]

\theoremstyle{remark}
\newtheorem{remark}{Remark}[section]
\theoremstyle{remark}

\theoremstyle{definition}
\newtheorem{definition}{Definition}[section]
\newtheorem{example}{Example}[section]

\newcommand{\wtilde}{\widetilde}
\newcommand{\E}{\mathsf{E}}
\newcommand{\GrGd}{\mathsf{GrGd}}
\newcommand{\TGr}{\mathsf{TGr}}
\newcommand{\Cov}{\mathsf{Cov}}
\newcommand{\TC}{\mathsf{TC}}
\newcommand{\Cat}{\mathsf{Cat}}
\newcommand{\Act}{\mathsf{Act}}

\newcommand{\Ker}{\mathsf{Ker}}

\newcommand{\C}{\mathsf{C}}
\newcommand{\XMod}{\mathsf{XMod}}

\begin{document}
\title{Coverings of internal groupoids and crossed modules in the category of groups with operations}

\author{H. Fulya Akız\thanks{E-mail : hfulya@gmail.com}}
\author{Nazmiye Alemdar\thanks{E-mail : nakari@erciyes.edu.tr}}
\author{Osman Mucuk\thanks{E-mail : mucuk@erciyes.edu.tr}}
\author{Tunçar Şahan\thanks{E-mail : tuncarsahan@aksaray.edu.tr}}
\affil{\small{Department of Mathematics, Erciyes University, Kayseri, TURKEY}}

\date{}

\maketitle
\begin{abstract}
In this paper we prove  some results on the covering morphisms of  internal groupoids. We also give  a result on the coverings of the crossed modules of groups with operations.
\end{abstract}

\\ {\bf Classification:} 18D35, 22A05, 20L05, 57M10
	
\section{Introduction}

A {\em groupoid}  is a small category in which each morphism is an   isomorphism \cite{Br,Ma}.  A {\em group-groupoid} is a group object in the category of groupoids \cite{BS}; equivalently, as is well known, it is an internal groupoid in the category of groups and, according to [25], it is an internal category in the category of groups.  An alternative name, quite generally used,  is  ``$2$-group'',
see for example \cite{baez-lauda-2-groups}.   Recently the notion of monodromy for  topological  group-groupoids was  introduced and investigated in  \cite{Mu-Be-Tu-Na}, and normality and quotients in group-groupoids were studied  in \cite{Mu-Tu-Na}.

A {\em crossed module} defined by  Whitehead in \cite{Wth1,Wth2} can be viewed as a
2-dimensional group \cite{BrLowDim} and  has  been  widely used  in: homotopy theory \cite{Br-Hi-Si};
the theory of identities among relations for group presentations \cite{BrownHubesshman};    algebraic $K$-theory \cite{Loday}; and homological algebra, \cite{Hubeshman, Lue}.  See \cite[p.~49]{Br-Hi-Si} for a discussion of the relation of crossed modules to crossed squares and so to homotopy 3-types.

In \cite[Theorem 1]{BS} Brown and Spencer proved that the category of
internal categories within the groups, i.e. group-groupoids,  is equivalent to the
category of crossed modules of groups. Then in \cite[Section 3]{Por}, Porter proved that a similar result holds for certain algebraic categories $\C$, introduced by Orzech \cite{Orz}, the  definition of which was adapted by him and called category of groups with operations. Applying Porter's result, the study of internal category theory in $\C$ was continued in the works of Datuashvili \cite{Wh} and \cite{Kanex}. Moreover, she developed the cohomology theory of internal categories, equivalently, crossed modules, in categories of groups with operations \cite{Coh, Cohtr}. In a similar way, the results of  \cite{BS} and  \cite{Por} enabled us to prove some properties of covering groupoids for internal groupoids.

On the other hand, there are some important  results on  the covering groupoids of  group-groupoids.  One is that the group structure of a group-groupoid lifts to its  some   covering groupoids, i.e.  if $G$ is a group-groupoid, $0$ is the identity element of the additive group  $G_0$ of objects and $p\colon (\widetilde{G},\tilde{0})\rightarrow (G,0)$ is a  covering morphism of groupoids whose  characteristic group is a subgroup of $G$, then $\widetilde{G}$  becomes a group-groupoid such that $p$ is a morphism of  group-groupoids \cite[Theorem 2.7]{Na-Mu}.  Another is that if $X$ is a topological group whose topology is semi-locally simply connected, then the category ${\Cov}_{\TGr}/X$ of covers of $X$ in the category of  topological  groups   and the category ${\Cov}_{\GrGd}/\pi(X)$ of covers of  $\pi(X)$ in the category of group-groupoids  are equivalent \cite[Proposition 2.3]{BM1}. The another one   is that for a group-groupoid  $G$  the category   ${\Cov}_{\GrGd}/G$ of  covers of $G$ in the category of   group-groupoids  is equivalent to the category $\Act_{\GrGd}/G$ of group-groupoid actions of $G$ on groups \cite[Proposition 3.1]{BM1}. The final one is that if $X$ is a topological group, then the restriction $d_1\colon St_{\pi (X)} e\rightarrow X$ of the  final point map to the star at the identity $e\in X$  is a crossed module of groups and the category ${\Cov}_{\GrGd}/\pi(X)$ of covers of $\pi(X)$ in the category of  group-groupoids  and the category of covers of  $St_{\pi (X)} e\rightarrow X$ in the category of crossed modules within groups are equivalent \cite[Corollary 4.3]{BM1}.

The object of this paper is to prove that the above-mentioned results can be generalized to a wide class of algebraic categories, which include categories of groups, rings, associative algebras, associative commutative algebras, Lie algebras, Leibniz algebras, alternative algebras and others. These are conveniently handled by working in a category $\C$.   For the first result we prove that if  $G$ is an internal groupoid in $\C$   and $p\colon (\widetilde{G},\tilde{0})\rightarrow (G,0)$ is a covering morphism of groupoids whose characteristic group is closed under the group operations of $G$, then  $\widetilde{G}$  becomes an internal groupoid  such that $p$ is a morphism of internal groupoids.   For the second result  we prove that   if $X$ is a topological group with operations whose topology is semi-locally simply connected, then the category ${\Cov}_{\TC}/X$ of covers of $X$ in the category of topological groups with operations  and the category ${\Cov}_{\Cat(\C)}/\pi(X)$ of  covers of $\pi (X)$ in the category $\Cat(\C)$ of internal categories, equivalently, internal
groupoids in $\C$ (see the note after Definition \ref{Internalgpd}),  are equivalent. For the third one  we prove that if  $G$ is an internal groupoid in $\C$, then the category ${\Cov}_{\Cat(\C)}/G$ of covers of $G$ in the category of internal groupoids in $\C$ is equivalent to the category  $\Act_{\Cat(\C)}/G$ of internal groupoid  actions of $G$ on groups with operations. Finally we prove that if $G$ is an internal groupoid in $\C$ and $\alpha\colon A\rightarrow B$ is the crossed module corresponding to $G$,  then  the category ${\Cov}_{\Cat(\C)}/G$  and the category ${\Cov}_{\XMod}/(\alpha\colon A\rightarrow B)$ of  covering  crossed modules of  $\alpha\colon A\rightarrow B$ are  equivalent.

\section{Covering morphisms of groupoids}
As is defined in  \cite{Br,Ma}, a  groupoid $G$ has a set $G$ of morphisms, which we call just
{\em elements} of $G$, a set $G_0$ of {\em objects} together with maps
$d_0, d_1\colon G\rightarrow G_0$ and $\epsilon \colon G_0 \rightarrow
G$
such that $d_0\epsilon=d_1\epsilon=1_{G_0}$. The
maps
$d_0$, $d_1$ are called {\em initial} and {\em final} point maps, respectively,
and the map $\epsilon$ is called the {\em object inclusion}.
If $a,b\in G$ and $d_1(a)=d_0(b)$, then the {\em composite}
$a\circ b$ exists such that $d_0(a\circ b)=d_0(a)$ and $d_1(a\circ b)=d_1(b)$. So
there exists a partial composition  defined by
$ G_{d_1}\times_{d_0} G\rightarrow G, (a,b)\mapsto a\circ b$, where
$G_{d_1}\times_{d_0} G$ is the pullback of ${d_1}$ and ${d_0}$.
Further, this partial composition is associative, for $x\in G_0$ the element $\epsilon (x)$  acts as the identity, and each element $a$ has an inverse $a^{-1}$ such
that ${d_0}(a^{-1})={d_1}(a)$, ${d_1}(a^{-1})={d_0}(a)$, $a\circ a^{-1}=\epsilon {d_0}(a)$ and
$a^{-1}\circ a=\epsilon {d_1}(a)$. The map $G\rightarrow G$, $a\mapsto a^{-1}$
is called the {\em inversion}.

In a groupoid $G$, for $x,y\in G_0$ we write $G(x,y)$ for the set of all
morphisms with initial point $x$ and final point $y$. According to \cite{Br} $G$ is
{\em transitive} if for all $x,y \in G_0$, the set  $G(x,y)$ is not empty; for
$x \in G_0$  the {\em star}  of $x$ is defined as $\{a\in G\mid s(a)=x\} $ and denoted  as  $St_G x$; and the {\em object group} at $x$ is  defined as $G(x,x)$ and denoted as $G(x)$ .

Let $G$ and $H$ be groupoids. A {\em morphism} from $H$ to $G$ is
a pair of maps  $f\colon H\rightarrow G$ and $f_0\colon
H_0\rightarrow G_0$ such that $d_0 f=f_0 d_0$, $d_1
f=f_0 d_1$, $f\epsilon=\epsilon f_0$ and $f(a\circ b)=f(a)\circ f(b)$ for all $(a,b)\in
H_{d_1}\times_{d_0} H$. For such a morphism we simply write $f\colon
H\rightarrow G$.

Let $p\colon\wtilde G\rightarrow G$ be a morphism of groupoids. Then $p$ is
called a {\em covering morphism} and $\widetilde{G}$  a {\em  covering groupoid} of $G$ if for
each $\tilde x\in {\wtilde G}_0$ the restriction  $St_{\widetilde{G}} {\tilde{x}}\rightarrow St_G{p(\tilde x)}$  is  bijective. A covering morphism $p\colon \widetilde{G}\rightarrow
G$ is called {\em transitive } if both $\widetilde{G}$ and $G$ are
transitive.  A transitive covering morphism $p\colon\wtilde G\rightarrow G$
is called {\em universal} if $\wtilde G$ covers every cover of $G$, i.e.  if
for every covering morphism $q\colon \widetilde{H}\rightarrow G$ there is a unique
morphism of groupoids $\tilde{p}\colon \wtilde G\rightarrow \widetilde{H}$ such that $q\tilde{p}=p$
(and hence $\tilde{p}$ is also a covering morphism), this is equivalent to that for
$\tilde{x}, \tilde{y}\in O_{\wtilde G}$ the set $\wtilde{G}(\tilde x, \tilde y)$
has not more than one element.

A morphism  $p\colon
(\widetilde{G},\tilde{x})\rightarrow (G,x)$ of pointed groupoids is called a {\em
	covering morphism} if the
morphism $p\colon \widetilde{G}\rightarrow G$ is a covering
morphism.  Let $p\colon (\wtilde G,\tilde{x})\rightarrow (G,x)$ be a covering morphism of
groupoids  and  $f\colon (H,z)\rightarrow (G,x)$ a morphism of
groupoids. We say $f$ lifts to $p$ if there exists a unique
morphism $\tilde f\colon (H,z)\rightarrow (\wtilde G,\tilde{x})$
such that $f=p\tilde f$. For any groupoid morphism $p\colon \widetilde{G}\rightarrow G$ and
an object $\tilde{x}$ of $\widetilde{G}$ we call  the subgroup
$p(\widetilde{G}(\tilde{x}))$ of $G(p\tilde{x})$ the {\em
	characteristic group} of $p$ at $\tilde{x}$.

The following result gives a criterion on the liftings of
morphisms  \cite[10.3.3]{Br}.

\begin{theorem}\label{Liftcov} Let $p\colon(\wtilde
	G,\tilde{x})\rightarrow (G,x)$ be a covering morphism of
	groupoids and  $f\colon (H,z)\rightarrow (G,x)$ a morphism such that
	$H$ is transitive. Then the morphism $f\colon (H,z)\rightarrow (G,x)$
	lifts to a morphism $\tilde f\colon (H,z)\rightarrow
	(\widetilde{G},\tilde{x})$  if and only if the characteristic
	group of $f$ is contained in that of $p$; and if this lifting
	exists, then it is unique.
\end{theorem}

As a result of Theorem  \ref{Liftcov} the  following
corollary is stated in \cite[10.3.4]{Br}.

\begin{corollary} \label{CorLiftcov} Let $p\colon (\widetilde{G},\tilde{x})\rightarrow (G,x)$ and
	$q\colon (\widetilde{H},\tilde{z})\rightarrow (G,x)$ be transitive
	covering morphisms with characteristic groups $C$ and $D$,
	respectively. If $C\subseteq D$, then there is a unique covering
	morphism $r\colon (\widetilde{G},\tilde{x})\rightarrow
	(\widetilde{H},\tilde{z})$  such that $p=qr$. If $C=D$, then $r$
	is an isomorphism.\end{corollary}

The action of a groupoid on a set is defined in \cite[p.373]{Br} as follows.
\begin{definition} \label{gpdactionSet} Let $G$ be a groupoid. An {\em action of $G$ on a set } consists
	of a set $X$, a function $\theta\colon X\rightarrow G_0$ and a
	function $\varphi\colon X_{\theta}\times_{d_0}G\rightarrow X,
	(x,a)\mapsto xa$ defined on the pullback  $X_{\theta}\times_{d_0} G$ of $\theta$ and $d_0$ such that
	
	\begin{enumerate}
		\item  $\theta(xa)=d_1(a)$ for $(x,a)\in X_{\theta}\times_{d_0} G$;
		
		\item  $x(a\circ b)=(xa)b$ for  $(a,b)\in G_{d_1}\times_{d_0}G$ and $(x,a)\in X_{\theta}\times_{d_0} G $;
		
		\item $x\epsilon({\theta(x))}=x$ for $x\in X$.
		
	\end{enumerate}
\end{definition}

According to \cite{Br}, given such an action, the {\em action groupoid } $G\ltimes X$ is
defined to be the groupoid with object set $X$ and with  elements of
$(G\ltimes X)(x,y)$ the pairs $(a,x)$ such that $a\in
G(\theta(x),\theta(y))$ and $xa=y$. The groupoid composite is
defined to be
\[  (a,x)\circ (b,y)=(a\circ b,x)\] when $y=xa$.

The following result is from \cite[10.4.3]{Br}. We need some details of its proof in the proofs of Theorem  \ref{IGdactsonX} and  Theorem \ref{groupoidgrouplift}.

\begin{theorem} \label{Theoactiongpdcover} Let $x$ be an object of a transitive  groupoid $G$,and let $C$ be a \mbox{subgroup} of the object group $G(x)$. Then there exists a covering morphism $q\colon(\widetilde{G}_C,\tilde{x})\rightarrow (G,x)$ with characteristic group $C$.
\end{theorem}

\begin{proof}
	We give a  sketch proof for a technical method: Let $X$ be the set of  cosets
	$C\circ a=\{c\circ a\mid c\in\ C\}$ for $a$ in $St_G x$. Let $\theta\colon
	X\rightarrow G_0$ be a map, which sends  $C\circ a$ to the final point of $a$. The function $\theta$ is well defined because  if $C\circ a=C\circ b$ then  $t(a)=t(b)$.  The groupoid $G$ acts on $X$ by \[\varphi\colon X_{\theta}\times_{d_0}G\rightarrow X, (C \circ a,
	g)\mapsto C\circ (a\circ  g).\] The required groupoid $\widetilde{G}_C$ is
	taken to be the action groupoid $G\ltimes X$. Then the  projection
	$q\colon \widetilde{G}_C\rightarrow G$ given on objects by
	$\theta\colon X\rightarrow G_0$ and on elements by $(g,C \circ a)\mapsto
	g$, is a covering morphism of groupoids and has the characteristic
	group $C$.  Here  the groupoid composite on $\widetilde{G}_C$ is defined by
	\[(g,C\circ a)\circ(h,C\circ b)=(g\circ h,C\circ a)\]
	whenever $C\circ b=C\circ a \circ g$. The required object
	$\tilde{x}\in \widetilde{G}_C$ is the coset $C$.
\end{proof}

\section{Groups with operations and internal categories}
The idea  of the definition of categories of groups with
operations comes from Higgins \cite{Hig} and Orzech \cite {Orz};
and the definition below is from Porter \cite{Por} and Datuashvili \cite[p.~21]{Tamar}, which is adapted from Orzech \cite {Orz}.

From now on  $\C$  will be a category of groups with a set  of  operations $\Omega$ and with a set $\E$  of identities such that $\E$ includes the group laws, and the following conditions hold: If $\Omega_i$ is the set of $i$-ary operations in $\Omega$, then

(a) $\Omega=\Omega_0\cup\Omega_1\cup\Omega_2$;

(b) The group operations written additively $0,-$ and $+$ are
the  elements of $\Omega_0$, $\Omega_1$ and
$\Omega_2$ respectively. Let $\Omega_2'=\Omega_2\backslash \{+\}$,
$\Omega_1'=\Omega_1\backslash \{-\}$ and assume that if $\star\in
\Omega_2'$, then $\star^{\circ}$ defined by
$a\star^{\circ}b=b\star a$ is also in $\Omega_2'$. Also assume
that $\Omega_0=\{0\}$;

(c) For each   $\star \in \Omega_2'$, $\E$ includes the identity
$a\star (b+c)=a\star b+a\star c$;

(d) For each  $\omega\in \Omega_1'$ and $\star\in \Omega_2' $, $\E$
includes the identities  $\omega(a+b)=\omega(a)+\omega(b)$ and
$\omega(a)\star b=\omega(a\star b)$.

A category satisfying the conditions (a)-(d) is called a {\em category of groups with operations}.

A {\em  morphism} between any two objects of $\C$ is a group homomorphism, which preserves the operations from $\Omega_1'$ and $\Omega_2'$.

\begin{remark}  The set $\Omega_0$ contains exactly one element, the
	group identity; hence for instance the category of associative rings with unit is not a category of  groups with operations.\end{remark}

\begin{example} The categories of  groups, rings generally  without identity, $R$-mo\-du\-les,  associative, associative commutative, Lie, Leibniz,
	alternative algebras are examples of categories of   groups with operations.
\end{example}

If $A$ and $B$ are objects of $\C$ an {\em extension} of $A$ by $B$ is
an exact sequence
\[0\longrightarrow A\stackrel{\imath}\longrightarrow E\stackrel{p}\longrightarrow B\longrightarrow 0\]
in which $p$ is surjective and $\imath$ is the kernel of $p$.  It is {\em split } if there is a morphism $s\colon  B \to E$ such
that $p s = \imath d _B$.  A split extension of $B$ by $A$ is called a {\em  $B$-structure} on $A$.  Given  such a $B$-structure on $A$, we  get the actions of $B$ on $A$ corresponding to the operations in $\C$. For any $b\in B$, $a\in A$ and $\star\in \Omega'_2$ we have the actions called {\em derived actions} by Orzech \cite[p.~293]{Orz}
\begin{align*}
	b\cdot a & = s(b)+a-s(b),\\
	b\star a  & = s(b)\star a.
\end{align*}

\begin{theorem}[{\cite[Theorem 2.4]{Orz}}]\label{Derivedaction}  A set of actions $($one for each operation in $\Omega_2)$ is a set of
	derived actions if and only if the semidirect product $B \ltimes
	A$ with underlying set $B \times A$ and operations
	\begin{align*}
		(b', a') + (b, a) &= (b' + b, a'\cdot b+a), \\
		(b', a') \star (b, a) &= (b' \star b, b' \star a + a' \star b + a' \star a)
	\end{align*}
	is an object in $\C$.
\end{theorem}

\begin{definition}[{\cite[Definition 1.5]{Orz}}]   Let $X$ be an object of $\C$. A subobject  $A$  of $X$ is called an {\em ideal} if it is the kernel of some morphism.
\end{definition}

The concept of ideal is characterized as follows.

\begin{theorem}[{\cite[Theorem 1.7]{Orz}}]  Let $A$ be a subobject of the object $X$.  Then  $A$ is an ideal of $X$ if and only if the following conditions hold:
	\begin{enumerate}
		\item $A$ is a normal subgroup of $X$;
		\item  $a\star x\in A$ for   $a\in A$, $x\in X$ and $\star\in\Omega_2'$.
	\end{enumerate}
\end{theorem}

We define the category of topological groups with operations as follows.

In the rest of the paper $\TC$ will denote the category of topological groups with a set $\Omega$ of continuous operations  and with a set $\E$ of identities such  that $\E$ includes the group laws such that the conditions (a)-(d) of Section 2 are satisfied.

Such a category is called a {\em  category of topological groups with operations}.

A {\em  morphism} between any two objects of $\TC$ is a continuous group homomorphism, which preserves the operations in $\Omega_1'$ and $\Omega_2'$.

The categories of topological groups,  topological rings  and  topological  $R$-mo\-du\-les are examples of categories of  topological groups with operations.

The internal category in $\C$ is defined in \cite{Por} as follows. We  follow the notations of Section 1 for groupoids.
\begin{definition}\label{Internalgpd}  An {\em internal category} $C$  in $\C$  is a category in which the initial and final point maps $d_0,d_1\colon C\rightrightarrows
	C_0$, the object inclusion map  $\epsilon\colon C_0\rightarrow C$
	and the partial composition $\circ\colon C_{d_1}\times_{d_0}
	C\rightarrow C,(a,b)\mapsto a\circ b$ are the morphisms in the category
	$\C$.\end{definition}

Note that since $\epsilon$ is a morphism in $\C$,
$\epsilon(0)=0$ and that the operation $\circ$ being a morphism
implies that for all $a,b,c,d\in C$ and $\star\in \Omega_2$,
\begin{align}\label{1}
	(a\star b)\circ(c\star d)=(a\circ c)\star(b\circ d)
\end{align}
whenever
one side makes sense. This is called the {\em interchange law} \cite{Por} .

We also note from \cite{Por} that  any internal category in
$\C$ is an internal groupoid since, given $a\in C$,
$a^{-1}=\epsilon d_1(a)-a+\epsilon d_0(a)$ satisfies $a^{-1}\circ a=\epsilon d_1(a)$ and
$a\circ a^{-1}=\epsilon d_0(a)$. So we use the term {\em internal groupoid}  rather than  internal category and write  $G$ for an internal groupoid.  For the category of internal groupoids in $\C$ we use   the same notation $\Cat(\C)$ as in \cite{Por}.  Here a {\em morphism}  $f\colon H\rightarrow G$ in $\Cat(\C)$  is the morphism of underlying groupoids and a morphism in $\C$.

In particular if $\C$ is the  category of groups, then an internal groupoid  $G$ in $\C$
becomes a  group-groupoid and in the case where $\C$ is the category of rings, an internal groupoid in $\C$ is a ring object in the category of groupoids \cite{Mu2}.

\begin{remark} \label{Remcorfromdef}  We emphasize the  following points from  Definition \ref{Internalgpd}:
	
	\item (i)
	By Definition \ref{Internalgpd} we know  that in an internal groupoid $G$ in $\C$, the initial and final point maps $d_0$ and $d_1$, the object inclusion  map $\epsilon$ are the morphisms in $\C$ and the interchange law  (\ref{1}) is satisfied. Therefore in an internal groupoid $G$,  the unary operations are endomorphisms of the underlying groupoid of $G$ and the binary operations are  morphisms from the underlying groupoid of  $G\times G$ to the one of $G$.

	\item  (ii) Let $G$ be an internal groupoid in $\C$ and $0\in G_0$ the identity element. Then $\Ker d_0=St_G0$, called in \cite{Br} the {\em transitivity component} or {\em connected component} of $0$, is also an internal groupoid which is also an ideal of $G$.
\end{remark}

The following example plays the key rule in the proofs of  Corollary \ref{Cortopgpoper} and  Theorem \ref{eqcat}.

\begin{example}\label{Teotopgroupwithoper}  \label{fungdtgp}
	If $X$ is  an object of $\TC$, then the fundamental groupoid $\pi(X)$ is an internal groupoid. Here the operations on    $\pi (X)$ are induced  by those of  $X$.   The details  are straightforward.
\end{example}

\begin{definition}\label{Defofcovemorpigd}  A morphism  $f\colon
	H\rightarrow G$ of internal groupoids in $\C$ is called a {\em cover}
	(resp. {\em universal cover}) if it is a covering morphism
	(resp. universal covering morphism) on the underlying groupoids.
\end{definition}

Since by Remark \ref{Remcorfromdef} (ii) for an internal groupoid $G$ in $\C$, the star  $St_G 0$ is also an internal groupoid, we have that   if $f\colon H\rightarrow G$ is a covering morphism of internal groupoids, then the restriction of $f$ to the  stars $St_H0\rightarrow St_G0$  is an isomorphism in $\C$.

A morphism $p\colon \widetilde{X}\rightarrow X$ in $\TC$ is called a {\em covering morphism} of topological groups with operations if it is a covering map on the underlying space.
\begin{example} \label{Examcovmorpintgpd}  If $p\colon \widetilde{X}\rightarrow X$ is a covering morphism
	of  topological groups with operations, then the induced morphism
	$\pi(p)\colon \pi(\widetilde{X})\rightarrow \pi(X)$ is a covering
	morphism of internal groupoids in $\C$.
\end{example}

In  \cite{BM1, Mu1}  an action of a group-groupoid $G$ on a group $X$ is defined.  We now  define an action of an internal groupoid on a group with operations as follows.

\begin{definition} Let  $G$ be an internal groupoid in $\C$ and $X$ an object of $\C$.  If the underlying groupoid of $G$ acts on the underlying set of  $X$ in the sense of  Definition~\ref{gpdactionSet} so that the maps  $\theta\colon X\rightarrow G_0$ and $\varphi\colon X_\theta \times _{d_0} G\rightarrow X, (x,a)\mapsto xa$ in the groupoid action are morphisms in $\C$, then we say that the internal groupoid $G$ {\em acts} on the group with operations $X$ via $\theta$.
\end{definition}
We write $(X,\theta,\varphi)$ for an action. Here note that  $\varphi\colon X_\theta\times _{d_0} G\rightarrow X, (x,a)\mapsto xa$ is a morphism in $\C$ if and only if
\begin{align}\label{2}
	(x\star y)(a\star b)=(xa)\star (yb)
\end{align}
for $x,y\in X$; $a,b\in G$ and $\star\in \Omega_2$  whenever  one side is defined.

\begin{example}\label{Examaction} Let  $G$ and $\widetilde{G}$ be internal groupoids in $\C$ and let   $p\colon \widetilde{G}\rightarrow G$   be a covering  morphism of internal groupoids. Then the internal groupoid  $G$ acts on the group with operations $X={\widetilde{G}}_0$ via $p_0\colon X\rightarrow G_0$ assigning to $x\in X$ and $a\in St_Gp(x)$ the target of the unique lifting $\tilde{a}$   in $\widetilde{G}$ of $a$ with source $x$. Clearly, the underlying groupoid of $G$ acts on the underlying set   and  by evaluating the  uniqueness of the lifting,  condition (\ref{2})  is satisfied for  $x,y\in X$ and  $a,b\in G$ whenever one side is defined.
\end{example}

We use  Theorem \ref{Theoactiongpdcover} to prove the following result for internal groupoids.
\begin{theorem} \label{IGdactsonX} Let  $G$ be an internal groupoid in $\C$ with transitive underlying groupoid, and  $0\in G_0$ the identity element of  the additive operation. Let  $G(0)$  be the object group at $0\in G_0$, which is a group with operations and  $C$ a subobject  of  $G(0)$. Suppose that  $X$ is  the set of  cosets $C\circ a=\{c\circ a\mid c\in\ C\}$ for $a$ in $St_G 0$.  Then $X$ becomes a group with operations such that  the internal  groupoid $G$ acts on  $X$ as a group with operations.
\end{theorem}
\begin{proof} Define $2$-ary operations  on  $X$ by
	\[(C\circ a)\tilde{\star}(C\circ b)=C\circ (a\star b)\] for $C\circ a,~ C\circ b\in X$.  We now prove that these  operations are well defined.  If  $C\circ a,~C\circ b\in X$, then  $a,b\in St_G 0$, $a\star b\in St_G 0$ and so  $C\circ (a\star b)\in X$.  Further if $C\circ a=C\circ a'$ and $C\circ b=C\circ b'$, then
	$a'\circ a^{-1}, b'\circ b^{-1} \in C$ and  by the interchange
	law  (\ref{1}) in $G$     \[(a'\star b')\circ (a\star b)^{-1}= (a'\star b')\circ (a^{-1}\star b^{-1})=(a'\circ a^{-1})\star (b'\circ b^{-1}).\]  Since  $C$ is a subobject, we have that  $(a'\star b')\circ (a\star b)^{-1}\in C$ and therefore
	$C\circ(a\star b)=C\circ(a'\star b')$.
	
	Define $1$-ary operations  on  $X$ by
	\[\tilde{\omega}(C\circ a)=C\circ \omega(a).\] If $C\circ a=C\circ b$, then $b\circ a^{-1}\in C$ and \[\omega(b\circ a^{-1})=\omega (b)\circ \omega (a^{-1})=\omega(b)\circ (\omega a)^{-1}\in C.\]
	So $C\circ \omega(a)=C\circ \omega(b)$  and hence these operations are well defined.  Since $G$ is a group with operations it has a set  $E$  of identities including group axioms.  So, the same type of identities including group axioms  and the other axioms (a)-(d)  of  Section 2  are satisfied for  $X$.  Therefore $X$ becomes a group with operations. We know from the proof of Theorem \ref{Theoactiongpdcover} that the underlying groupoid $G$ acts on the set $X$. In addition to these, it is straightforward to show  that by the interchange law (\ref{1}) in $G$, the condition
	\[((C\circ a)\tilde{\star}(C\circ b))(g\star h)=((C\circ a)\circ g)\tilde{\star} ((C\circ b)\circ h)\]
	is satisfied for $g,h\in G$ whenever the right side is defined. \end{proof}

\begin{theorem}\label{Teointernalgpd}  Let $G$ be an internal category in $\C$ and $X$ an object of $\C$. Suppose that the internal groupoid $G$ acts on the group with operations $X$. Then the action groupoid $G\ltimes X$   defined in Section 1 becomes  an internal groupoid in $\C$ such that the projection $p\colon G\ltimes X\rightarrow G$ is a morphism of internal groupoids.\end{theorem}
\begin{proof} We first prove that the action groupoid $\widetilde{G}=G\ltimes X$ is a group with operations in $\C$.   For this,  $1$-ary operations are defined by $\tilde{\omega}(a,x)=(\omega(a),\omega(x))$ for $\Omega_1$ and  $2$-ary  operations  are  defined by $(a,x)\tilde{\star} (b,y)=(a\star b,x\star y)$ for $\star\in \Omega_2$.    Then the  axioms (a)-(d) of Section 2  for  these operations defined on $\widetilde{G}$ are satisfied.
	Since $G$ and $X$  are  objects of $\C$, they have the same type of identities including group axioms.  So, the same type of identities including group axioms are satisfied for  $\widetilde{G}$.
	
	The initial and final point maps $d_0,d_1\colon \widetilde{G}\rightrightarrows
	X$, the object inclusion map  $\epsilon\colon X\rightarrow \widetilde{G}$
	and the partial composition $\circ\colon \widetilde{G}_{d_1}\times_{d_0}
	\widetilde{G}\rightarrow \widetilde{G},(\tilde{a},\tilde{b})\mapsto \tilde{a}\circ \tilde{b}$ are the morphisms in the category $\C$ because the same maps for $G$ are  morphisms in $\C$.
	
	By the  interchange law (\ref{1}) in $G$,  the partial composition  $\circ\colon \widetilde{G}_{d_1}\times_{d_0}
	\widetilde{G}\rightarrow \widetilde{G},(\tilde{a},\tilde{b})\mapsto \tilde{a}\circ \tilde{b}$ is a morphism in $\C$ and, by condition (\ref{2}) in $G$, the final point map $d_0\colon \widetilde{G}\rightarrow
	X$ becomes a morphism in $\C$.  The rest of the proof is straightforward.
\end{proof}

\begin{definition} Let $\wtilde G$  be   a groupoid,  $G$ an internal  groupoid in
	$\C$ and $0\in G_0$  the   identity of the additive operation.  Suppose that   $p\colon(\wtilde G,\tilde{0})\rightarrow (G,0)$ is a
	covering morphism of groupoids.  We say that the internal groupoid structure of $G$ {\em lifts} to $\wtilde G$ if $\wtilde G$ is an internal groupoid in $\C$ such that $p$ is a morphism of internal groupoids in
	$\C$.
\end{definition}

We use Theorems \ref{IGdactsonX} and \ref{Teointernalgpd}  to prove that the internal groupoid structure of an internal category  $G$ lifts  to  a covering groupoid.

\begin{theorem}\label{groupoidgrouplift} Let $\wtilde G$ be a groupoid and $G$ an internal groupoid in $\C$ whose underlying groupoid is transitive.  Suppose that  $p\colon(\wtilde G, \tilde{0})\rightarrow (G,0)$ is a covering morphism of underlying groupoids such that the characteristic group $C$ of $p$ is a subobject of $G(0)$. Then the internal groupoid structure of
	$G$ lifts to $\wtilde G$ with identity $\tilde 0$.
\end{theorem}
\begin{proof} Let $C$ be the characteristic group of $p\colon(\wtilde
	G,\tilde{0})\rightarrow (G,0)$.  Then by
	Theorem \ref{Theoactiongpdcover} we have a covering morphism of groupoids
	$q\colon(\wtilde{G}_{C},\tilde{x})\rightarrow (G,0)$ which has the
	characteristic group $C$. So by Corollary \ref{CorLiftcov} the
	covering morphisms $p$ and $q$ are isomorphic. Therefore  it is sufficient to prove that the internal groupoid  structure of  $G$ lifts to $\widetilde{G}_C= G\ltimes X$.  By Theorem \ref{IGdactsonX},  $X$ is a group with operations and  the internal groupoid $G$ acts on $X$. By Theorem \ref{Teointernalgpd}, the internal groupoid structure of $G$ lifts to $\widetilde{G}_C= G\ltimes X$. \end{proof}

From Theorem \ref{groupoidgrouplift} we obtain the following corollary.
\begin{corollary} \label{Cortopgpoper} Let  $X$ be an object of $\TC$   whose underlying space is connected. Suppose that $\wtilde X$ is  a simply connected topological space and   $p\colon\wtilde X\rightarrow X$ is a
	covering  map from $\wtilde X$ to the underlying topology of $X$.  Let $0$ be the identity
	element of the additive group of $X$ and  $\tilde 0\in\wtilde X$
	such that  $p(\tilde 0)=0$. Then $\wtilde X$ becomes a
	topological group with operations such that  $\tilde 0$  is the identity
	element of the group structure of $\wtilde X$ and  $p$ is a
	morphism of topological groups with operations.
\end{corollary}
\begin{proof}
	Since $p\colon\wtilde X\rightarrow X$ is a covering map, the induced morphism $\pi(p)\colon\pi(\wtilde X)\rightarrow \pi(X)$ becomes  a covering morphism of  groupoids  with a trivial  characteristic group.  Since $X$ is a topological group with operations,
	by Example  \ref{Teotopgroupwithoper} $\pi(X)$ is an  internal groupoid  and since $X$ is path connected, the groupoid $\pi(X)$ is transitive.   So by  Theorem \ref{groupoidgrouplift}, the internal groupoid
	structure of $\pi(X)$ lifts to   $\pi (\wtilde X)$.  So, we have the structure of a group with operations on  $\widetilde{X}$. By \cite[10.5.5]{Br} the  topology on $\widetilde{X}$  is the {\em lifted topology} obtained by  the covering morphism $\pi
	(p)\colon\pi(\wtilde X)\rightarrow \pi(X)$
	and the  group operations are continuous with this topology.   So, $\widetilde{X}$ becomes a  topological group with operations.
\end{proof}

\section{The equivalence of the categories}

Let $X$ be an object of $\TC$. So, by Example \ref{Teotopgroupwithoper}, $\pi(X)$ is an internal groupoid.  Then we have a  category
${\Cov}_{\TC}/X$ of covers of $X$ in the category $\TC$ of   topological groups with operations and a category
${\Cov}_{\Cat(\C)}/\pi(X)$ of covers of $\pi(X)$ in the category $\Cat(\C)$  of internal groupoids in  $\C$.

\begin{theorem}\label{eqcat} Let $X$ be an object of $\TC$  such that the underlying topology of $X$ has a simply connected cover. Then the categories
	${\Cov}_{\TC}/X$ and ${\Cov}_{\Cat(\C)}/\pi(X)$ are equivalent.
\end{theorem}
\begin{proof}
	Define a functor
	\[ \pi  \colon {\Cov}_{\TC}/X\rightarrow
	{\Cov}_{\Cat(\C)}/\pi(X) \] as follows: suppose that  $p\colon\wtilde
	X\rightarrow X$ is  a covering morphism of topological groups with operations. Then by Example \ref{Examcovmorpintgpd},  the induced morphism $\pi(p)\colon\pi
	(\wtilde X)\rightarrow \pi (X)$ is a morphism of internal groupoids
	and a covering morphism on the underlying groupoids which preserves the group operations. Therefore  $\pi(p)\colon\pi (\wtilde X)\rightarrow \pi (X)$
	becomes a covering morphism of internal groupoids.
	
	We now define another  functor
	\[ \eta \colon  {\Cov}_{\Cat(\C)} / \pi(X) \rightarrow {\Cov}_{\TC}/X\]
	as follows: suppose that  $q\colon\wtilde{G}\rightarrow
	\pi(X)$ is a covering morphism of  internal grou\-poids. By the lifted topology \cite[10.5.5]{Br} on $\wtilde
	X=\wtilde{G}_0$  there is an isomorphism $ \alpha\colon
	\wtilde{G}\rightarrow \pi(\widetilde{X})$ of groupoids  such that
	$p=O_q\colon \widetilde{X}\rightarrow X$ is a covering map and
	$q=\pi(p)~\alpha$.  Hence the group operations on
	$\wtilde{G}$ transport via $\alpha$ to $\pi(\widetilde{X})$ such that $\pi(\widetilde{X})$ is an internal groupoid.  So we have 2-ary  operations
	\[\tilde{\star}\colon \pi(\widetilde{X})\times \pi(\widetilde{X})\rightarrow \pi(\widetilde{X})\]
	such that $\pi(p)\circ \tilde{\star}=\pi(\star)\circ (\pi p\times \pi p)$, where $\star$'s are the 2-ary operations of $X$. By  \cite[10.5.5]{Br}  the operations $\tilde{\star}$ induce continuous 2-ary operations on $\widetilde{X}$. Similarly, there are 1-ary operations
	\[\tilde{\omega}\colon \pi(\widetilde{X})\rightarrow \pi(\widetilde{X})\]
	such that $(\pi p)\circ \tilde{\omega}=(\pi\omega)\circ \pi p$, where $\omega$'s are the 1-ary operations of $X$ and the operations $\tilde{\omega}$ induce continuous 1-ary operations on $\widetilde{X}$.  So $\widetilde{X}$ becomes a group with operations and by the lifted topology the operations are continuous.
	Hence  $\widetilde{X}$ becomes a topological group with operations.
	
	Since  the category $ {\Cov}_{\TGr}/X$ of covers of $X$ in the category of  topological groups is equivalent to the category ${\Cov}_{\GrGd}/\pi (X)$ of covers of  $\pi(X)$ in the category of   group-groupoids, by the
	following diagram the proof is completed
	\begin{equation*}
		\xymatrix{
			{\Cov}_{\TC}/X \ar[d] \ar[r]^>>>>{\pi} & {\Cov}_{\Cat(\C)}/\pi (X)  \ar[d] \\
			{\Cov}_{\TGr}/X   \ar[r]^>>>>{\pi}      & {\Cov}_{\GrGd}/\pi (X). }
	\end{equation*}
	The proof of theorem is complete.
\end{proof}

Let $G$ be an internal groupoid in $\C$.  Let    ${\Cov}_{\Cat(\C)}/G$  be the category of covers of $G$ in the category $\Cat(\C)$ of internal categories in $\C$. So the objects of ${\Cov}_{\Cat(\C)}/G$  are  the
covering morphisms $p\colon \widetilde{G}\rightarrow G$ over $G$ of internal groupoids and a morphism   from $p\colon \widetilde{G}\rightarrow G$ to $q\colon \widetilde{H}\rightarrow G$ is a
morphism $f\colon \widetilde{G}\rightarrow \widetilde{H}$ of internal groupoids, which becomes also a covering
morphism, such that $qf=p$.

Let  $\Act_{\Cat(\C)}/G$ be the category of internal groupoid actions of $G$ on  groups with operations. So, an object  of $\Act_{\Cat(\C)}/G$ is an  internal groupoid action $(X,\theta,\varphi)$ of $G$  and a morphism  from  $(X,\theta,\varphi)$ to  $(Y,\theta',\varphi')$ is a morphism   $f\colon X\!\rightarrow
X'$ of groups with operations such that $\theta=\theta'f$ and $f(xa)=(fx)a$ whenever $xa$ is defined, for any  objects $(X,\theta,\varphi)$ and  $(Y,\theta',\varphi')$ of  $\Act_{\Cat(\C)}/G$.

\begin{theorem} For an internal groupoid  $G$ in $\C$, the categories $\Act_{\Cat(\C)}/G$ and
	${\Cov}_{\Cat(\C)}/G$  are equivalent.\end{theorem}
\begin{proof} By Theorem \ref{Teointernalgpd}   for an internal groupoid  action $(X,\theta,\varphi)$ of $G$, we have a
	morphism $p\colon G\ltimes X\rightarrow G$ of internal groupoids, which is a covering morphism on underlying groupoids.  This gives a functor \[\Gamma\colon
	\Act_{\Cat(\C)}/G\rightarrow {\Cov}_{\Cat(\C)}/G.\]
	
	Conversely, if $p\colon \widetilde{G}\rightarrow G$ is a covering morphism of
	internal groupoids, then by Example \ref{Examaction} we have an internal groupoid action.
	In this way we define a functor  \[\Phi\colon
	{\Cov}_{\Cat(\C)}/G\rightarrow \Act_{\Cat(\C)}/G.\]
	
	The natural equivalences $\Gamma\Phi\simeq 1$ and $\Phi\Gamma\simeq 1$ follow.
\end{proof}

\section{Covers of crossed modules in groups with operations}

The conditions of a crossed module in groups with operations are formulated in \cite[Proposition 2]{Por} as follows.
\begin{definition} \label{Defcrosmod} A {\em crossed module} in $\C$ is  $\alpha\colon A\rightarrow B$ is a
	morphism in $\C$, where $B$ acts on $A$ (i.e. we have a derived action in $\C$)  with the conditions for any $b \in B$, $a, a'\in A$, and $\star\in\Omega_2'$:
	\begin{enumerate}\item[]
		\begin{enumerate}
			\item [CM1~~]$\alpha(b \cdot a) = b + \alpha(a) - b$;
			\item [CM2~~]$\alpha(a)\cdot a'=a+a'-a$;
			\item [CM3~~] $\alpha(a)\star a'=a\star a'$;
			\item [CM4~~] $\alpha(b\star a)=b\star\alpha(a)$ and  $\alpha(a\star b)=\alpha(a)\star b$.
		\end{enumerate}
	\end{enumerate}
\end{definition}

A {\em morphism} from  $\alpha\colon A\rightarrow B$ to $\alpha'\colon A'\rightarrow B'$   is a pair
$f_1\colon A\rightarrow A'$ and  $f_2\colon B\rightarrow B'$ of morphisms in $\C$ such that
\begin{enumerate}[label={ \arabic{*}.} , leftmargin=1cm]
	\item $f_2 \alpha(a)=\alpha' f_1(a)$,
	\item $f_1(b\cdot a)=f_2(x)\cdot f_1(a)$,
	\item $f_1(b\star a)=f_2(x)\star f_1(a)$
\end{enumerate}
for any $x\in B$, $a\in A$ and $\star\in\Omega_2'$.  So the category $\XMod$ of crossed modules in groups with operation is obtained.

The following theorem was proved  in \cite[Theorem1]{Por}.
\begin{theorem} \label{Theocatequivalencess}  The category $\XMod$ of crossed modules  and the category $\Cat(\C)$  of internal groupoids in $\C$ are equivalent.\end{theorem}

By Theorem \ref{Theocatequivalencess}, evaluating the  covering morphism of internal groupoids (De\-fi\-nition \ref{Defofcovemorpigd}) in terms of the corresponding morphism of  crossed modules in $\C$, we can  obtain the notion of a covering morphism  of crossed modules in $\C$. If $f\colon H\rightarrow G$ is a covering morphism of internal groupoids in $\C$ as defined in Definition \ref{Defofcovemorpigd} and $(f_1,f_2)$ is the morphism of crossed modules corresponding to $f$, then $f_1\colon A\rightarrow A'$ is an isomorphism in $\C$, where $A=St_H 0$, $A'=St_G 0$ and $f_1$ is the restriction of $f$.  Therefore we call a morphism $(f_1,f_2)$  of crossed modules  from $\alpha\colon A\rightarrow B$ to $\alpha'\colon A'\rightarrow B'$ in $\C$ as a {\em cover} if  $f_1\colon A\rightarrow A'$ is an isomorphism in $\C$.

\begin{theorem} \label{Theocatequivalence}  Let  $G$ be an internal groupoid in $\C$ and    $\alpha\colon A\rightarrow B$  the  crossed module in $\C$  corresponding to $G$.  Let ${\Cov}_{\Cat(\C)}/G$ be the category of covers  of $G$ in the category $\Cat(\C)$ of  internal groupoids in $\C$ and let ${\Cov}_{\XMod}/(\alpha\colon A\rightarrow B)$ be the category of covers of $\alpha\colon A\rightarrow B$ in $\C$.  Then the categories ${\Cov}_{\Cat(\C)}/G$ and ${\Cov}_{\XMod}/(\alpha\colon A\rightarrow B)$ are equivalent. \end{theorem}
\begin{proof} The proof is obtained from Theorem \ref{Theocatequivalencess} and therefore the details are omitted.
\end{proof}

Let $X$ be an object of $\TC$. Then by Example \ref{Teotopgroupwithoper}, $\pi(X)$ is an internal groupoid and therefore  $d_1\colon St_{\pi(X)} 0\rightarrow X$ is a crossed module in  $\C$.  So as a result of Theorems \ref{eqcat}  and \ref{Theocatequivalence} we can obtain the following corollary.
\begin{corollary} If $X$ is a topological group with operations in $\TC$ whose underlying topology has a simply connected cover, then the category ${\Cov}_{\TC}/X$ of covers of $X$ in the category $\TC$ of topological groups with operations  and the category  ${\Cov}_{\XMod}/ (d_1\colon St_{\pi (X)} 0\rightarrow X)$ of covers of the crossed module $d_1\colon St_{\pi (X)} 0\rightarrow X$ in $\C$ are equivalent.
\end{corollary}

\end{document}